\documentclass[11pt]{article}

\usepackage{a4,amsfonts,exscale}

\long\def\forget#1{}

\usepackage{amsmath}
\usepackage{amssymb}
\usepackage{amscd}
\usepackage{textcomp}

\usepackage{amsthm}
\theoremstyle{plain}
\newtheorem{theorem}{Theorem}[section]

\newtheorem{proposition}[theorem]{Proposition}

\theoremstyle{definition}
\newtheorem{definition}[theorem]{Definition}

\newtheorem{bigexample}[theorem]{Example}

\theoremstyle{remark}

\newtheorem{remark}[theorem]{Remark}

\usepackage{xy}
\xyoption{all}

\newdir^{ (}{{}*!/-3pt/\dir^{(}}    
\newdir^{  }{{}*!/-3pt/\dir^{}}    
\newdir_{ (}{{}*!/-3pt/\dir_{(}}    
\newdir_{  }{{}*!/-3pt/\dir_{}}

\newcounter{zahl}

\def\theenumi{(\alph{enumi})}

\def\p@enumii{\theenumi}

\newcommand{\DS}{\displaystyle}
\newcommand{\TS}{\textstyle}

\newcommand{\SSC}{\scriptscriptstyle}

\DeclareMathOperator{\Frob}{Frob}

\DeclareMathOperator{\Grass}{Grass}

\DeclareMathOperator{\Hom}{Hom}

\DeclareMathOperator{\Lie}{Lie}

\DeclareMathOperator{\Var}{V}

\newcommand{\alg}{{\rm alg}}
\newcommand{\an}{{\rm an}}

\newcommand{\cris}{{\rm cris}}

\DeclareMathOperator{\id}{\,id}

\DeclareMathOperator{\ord}{ord}

\newcommand{\rig}{{\rm rig}}
\DeclareMathOperator{\rk}{rk}

\renewcommand{\phi}{\varphi}
\renewcommand{\epsilon}{\varepsilon}

\usepackage{amsfonts}
\newcommand{\BOne} {{\mathchoice{\hbox{\rm1\kern-2.7pt l\kern.9pt}}
                              {\hbox{\rm1\kern-2.7pt l\kern.9pt}}
                              {\hbox{\scriptsize\rm1\kern-2.3pt l\kern.4pt}}
                              {\hbox{\scriptsize\rm1\kern-2.4pt l\kern.5pt}}}}

\newcommand{\BD}{{\mathbb{D}}}

\newcommand{\BF}{{\mathbb{F}}}

\newcommand{\BN}{{\mathbb{N}}}

\newcommand{\BQ}{{\mathbb{Q}}}

\newcommand{\BX}{{\mathbb{X}}}

\newcommand{\BZ}{{\mathbb{Z}}}

\newcommand{\bA}{{\mathbf{A}}}
\newcommand{\bB}{{\mathbf{B}}}

\newcommand{\bE}{{\mathbf{E}}}

\newcommand{\calD}{{\mathcal{D}}}

\newcommand{\CF}{{\mathcal{F}}}
\newcommand{\CG}{{\mathcal{G}}}

\newcommand{\CM}{{\mathcal{M}}}
\newcommand{\CN}{{\mathcal{N}}}
\newcommand{\CO}{{\mathcal{O}}}

\newcommand{\CS}{{\mathcal{S}}}

\newcommand{\rbij}{\mbox{\mathsurround=0pt \;$-\hspace{-0.75em}\stackrel{\sim\quad}{\longrightarrow}$\;}}

\let\setminus\smallsetminus

\newcommand{\es}{\enspace}

\newcommand{\dual}{^{\SSC\lor}}

\newcommand{\wt}[1]{{\widetilde{#1}}}

\usepackage{ifthen}

\newcommand{\invlim}[1][]{\ifthenelse{\equal{#1}{}}
{\DS \lim_{\longleftarrow}}
{\DS \lim_{\underset{#1}{\longleftarrow}}}
}

\newcommand{\dirlim}[1][]{\ifthenelse{\equal{#1}{}}
{\DS \lim_{\longrightarrow}}
{\DS \lim_{\underset{#1}{\longrightarrow}}}
}

\newcommand{\dotBD}{\vbox{\hbox{\kern2pt\bf.}\vskip-4.5pt\hbox{$\BD$}}}

\def\?{\ 
???\ \immediate\write16{}
\immediate\write16{Warning: There was still a question mark . . . }
\immediate\write16{}}

\DeclareMathOperator{\Nilp}{\CN \!{\it ilp}}
\DeclareMathOperator{\Sets}{\CS \!{\it ets}}

\def\isoto{\rbij}

\newbox\mybox
\def\arrover#1{\mathrel{
       \setbox\mybox=\hbox spread 1.4em{\hfil$\scriptstyle#1$\hfil}
       \vbox{\offinterlineskip\copy\mybox
             \hbox to\wd\mybox{\rightarrowfill}}}}

\usepackage{amsthm}
\theoremstyle{plain}
\newtheorem{problem}[theorem]{Problem}

\begin{document}

\author{Urs Hartl
\footnote{The author acknowledges support of the Deutsche Forschungsgemeinschaft in form of DFG-grant HA3006/2-1}
}

\title{On period spaces for $p$-divisible groups}

\maketitle

\begin{abstract}
\noindent
In their book 
Rapoport and Zink constructed rigid analytic period spaces for Fontaine's filtered isocrystals, and period morphisms from moduli spaces of $p$-divisible groups to some of these period spaces. We determine the image of these period morphisms, thereby contributing to a question of Grothendieck. We give examples showing that only in rare cases the image is all of the Rapoport-Zink period space.

\medskip

\begin{center}
\bfseries R\'esum\'e
\end{center}

\noindent 
Dans leur livre, Rapoport et Zink ont construits des espaces des p\'eriodes, rigides ana\-lytiques pour les isocristaux filtr\'es de Fontaine. Egalement ils ont construits des morphismes des p\'eriodes entre des espaces modulaires des groupes de Barsotti-Tate et certaines de leurs espaces des p\'eriodes. Dans cet article nous d\'eterminons l'image des morphismes des p\'eriodes, contribuant ainsi \`a une question de Grothendieck. Nous pr\'esentons des examples montrant que l'image ne co\"incide que rarement avec tout l'espace des p\'eriodes de Rapoport-Zink.

\medskip

\end{abstract}

\section{A question of Grothendieck}

 Fix a Barsotti-Tate group $\BX$ over $\BF_p^\alg$ of height $h$ and dimension $d$. Let $W:=W(\BF_p^\alg)$ be the ring of Witt vectors and let $K_0:=W[\frac{1}{p}]$. Let $\CO_K$ be a complete discrete valuation ring with residue field $\BF_p^\alg$ and fraction field $K$ of characteristic $0$. To every Barsotti-Tate group $X$ over $\CO_K$ with $\BX\cong X\otimes_{\CO_K}\BF_p^\alg=:X_{\BF_p^\alg}$ the theory of Grothendieck-Messing~\cite{Messing} associates an extension
\[
\xymatrix {
0\ar[r] & (\Lie X\dual)_K\dual \ar[r] & \BD(\BX)_K\ar[r] & \Lie X_K\ar[r] & 0
}
\]
where $\BD(\BX)_K$ is the crystal of Grothendieck-Messing evaluated on $K$. We denote by $\CF$  the Grassmannian of $(h-d)$-dimensional subspaces of $\BD(\BX)_{K_0}$.

\begin{problem} \label{ProblemGroth}
(A.~Grothendieck \cite[Remarque 3, p.\ 435]{Grothendieck2})
Describe the subset of $\CF$ formed by the points $(\Lie X\dual)_K\dual$ where $X$ is any deformation of $\BX$ over any complete discrete valuation ring $\CO_K$ with residue field $\BF_p^\alg$ and fraction field $K$ of characteristic $0$.
\end{problem}

This problem is yet unsolved. However, a contribution was made by Rapoport and Zink~\cite{RZ} in the following way. Set $D:=\BD(\BX)_{K_0}$ and $\phi_D:=\BD(\Frob_\BX)_{K_0}$. Let $L_K\in\CF$ be a point defined over a complete, rank one valued extension $K$ of $K_0$ (which not necessarily is discrete). View $L_K$ as a $K$-subspace of $D_K:=D\otimes_{K_0}K$. One defines the \emph{Newton slope} $t_N(D,\phi_D,L_K):=\ord_p(\det\phi_D)$ and the \emph{Hodge slope} $t_H(D,\phi_D,L_K):=\dim_K L_K-\dim_K D_K$. Following Fontaine \cite{Fontaine79} and Rapoport-Zink~\cite[1.18]{RZ}, the point $L_K\in\CF$ is called \emph{weakly admissible} if
\begin{eqnarray*}
t_N(D,\phi_D,L_K)&=&t_H(D,\phi_D,L_K)\es=\es-d\qquad\text{and}\\
t_N(D',\phi_D|_{D'},L_K\cap D'_K)&\ge& t_H(D',\phi_D|_{D'},L_K\cap D'_K)
\end{eqnarray*}
for all $\phi_D$-stable $K_0$-subspaces $D'\subset D$. Let $\CF^\rig$ be the rigid analytic space, and $\CF^\an$ the $K_0$-analytic space in the sense of Berkovich~\cite{Berkovich,Berkovich2} associated with $\CF$. 

\begin{theorem} (Rapoport-Zink~\cite[Proposition 1.36]{RZ})\\
The set $\CF^\rig_{wa}:=\{\,L_K\in\CF^\rig:L_K \text{ is weakly admissible}\,\}$ is an open rigid analytic subspace. 
\end{theorem}

The space $\CF^\rig_{wa}$ is an example for the \emph{$p$-adic period domains} constructed more generally in \cite[Proposition 1.36]{RZ} for arbitrary filtered isocrystals. The proof of Rapoport and Zink even shows that $\CF^\rig_{wa}$ is the rigid analytic space associated to an open $K_0$-analytic subspace $\CF^\an_{wa}\subset\CF^\an$; see \cite[Proposition 1.3]{HartlRZ}. The period domain $\CF^\rig_{wa}$ contains the set of Grothendieck's problem~\ref{ProblemGroth}. To explain this let $\Nilp_W$ be the category of $W$-schemes on which $p$ is locally nilpotent. For an $S\in\Nilp_W$ we set $\bar S:=\Var(p)\subset S$.

\begin{theorem}\label{ThmRZ} (Rapoport-Zink~\cite[Theorem 2.16]{RZ})
The functor $\CG:\Nilp_W\to \Sets$
\begin{eqnarray*}
S&\longmapsto& \bigl\{\,\text{isomorphism classes of pairs }(X,\rho)\text{ where }X\text{ is a Barsotti-Tate}\\
&&\qquad\text{ group over $S$ and }\rho: \BX_{\bar S}\to X_{\bar S}\text{ is a quasi-isogeny}\,\bigr\}
\end{eqnarray*}
is pro-representable by a formal scheme locally formally of finite type over $W$.
\end{theorem}

To $\CG$ one can associate a rigid analytic space $\CG^\rig$ by Berthelot's construction~\cite{Berthelot96}, and also a $K_0$-analytic space $\CG^\an$. Rapoport and Zink~\cite[5.16]{RZ} construct a period morphism $\breve\pi_1^\rig:\CG^\rig\to\CF^\rig_{wa}$ as follows. By the theory of Grothendieck-Messing~\cite{Messing}, the universal Barsotti-Tate group $X$ over $\CG$ gives rise to an extension
\[
\xymatrix {
0\ar[r] & (\Lie X\dual)_{\CG^\rig}\dual \ar[r] & \BD(X_{\bar\CG})_{\CG^\rig}\ar[r] & \Lie X_{\CG^\rig}\ar[r] & 0
}
\]
of locally free sheaves on $\CG^\rig$. The quasi-isogeny $\rho:X_{\bar\CG}\to\BX_{\bar\CG}$ induces by the crystalline nature of $\BD(\,.\,)$ an isomorphism $\BD(\rho)_{\CG^\rig}:\BD(X_{\bar\CG})_{\CG^\rig}\isoto\BD(\BX)_{\CG^\rig}$ and the image $\BD(\rho)_{\CG^\rig}(\Lie X\dual)_{\CG^\rig}\dual$ defines a $\CG^\rig$-valued point of $\CF^\rig$. By \cite[5.27]{RZ} the induced morphism $\CG^\rig\to\CF^\rig$ factors through $\CF^\rig_{wa}$. This is the period morphism $\breve\pi_1^\rig$. The same construction also gives a $K_0$-analytic period morphism $\breve\pi_1^\an:\CG^\an\to\CF^\an$.

\begin{theorem}\label{ThmGrothQuest}(Rapoport-Zink, Colmez-Fontaine, Breuil, Kisin)\\
The set of Grothendieck's Problem~\ref{ProblemGroth} is contained in the subset $\CF^\rig_{wa}$. The latter also equals the image of the period morphism $\breve\pi_1^\rig:\CG^\rig\to\CF^\rig$ in the sense of sets.
\end{theorem}

\begin{proof}
Let $L_K=(\Lie X\dual)_K\dual$ be a point in Grothendieck's set, given by a Barsotti-Tate group $X$ over $\CO_K$ with $\BX\cong X_{\BF_p^\alg}$. Over $\CO_K/(p)$ this isomorphism lifts by rigidity \cite[Appendix]{Drinfeld2} to a quasi-isogeny $\rho:\BX_{\CO_K/(p)}\to X_{\CO_K/(p)}$ and $(X,\rho)$ gives a point of $\CG(\CO_K)$. By construction $\breve\pi_1^\rig(X,\rho)_K=L_K$. So the point $L_K$ belongs to the image of $\breve\pi_1^\rig$, which in turn lies in $\CF^\rig_{wa}$. It remains to show that every $K$-valued point $L_K\in\CF^\rig_{wa}$ for $K/K_0$ finite lies in the image of $\breve\pi_1^\rig$. By the theorem of Colmez-Fontaine~\cite{CF} the filtered isocrystal $(D,\phi_D,L_K)$ is admissible, i.e. arises from a crystalline $p$-adic Galois representation $V$. By Breuil~\cite[Theorem 1.4]{Breuil2} there is a Barsotti-Tate group $X$ over $\CO_K$ and an isomorphism $T_p X_K\otimes_{\BZ_p}\BQ_p\cong V$ if $p>2$. Kisin~\cite[Corollary 2.2.6]{Kisin06} extended Breuil's theorem to $p=2$ and reproved the Colmez-Fontaine Theorem. Fontaine's functor $D_\cris$ transforms the isomorphism $T_p X_K\otimes_{\BZ_p}\BQ_p\cong V$ into an isomorphism $\bar\rho_\ast:\bigl(\BD(X)_{K_0},\BD(\Frob_X)_{K_0},(\Lie X\dual)_K\dual\bigr)\isoto(D,\phi_D,L_K)$ of filtered isocrystals. This defines a quasi-isogeny $\bar\rho:\BX\to X_{\BF_p^\alg}$ which again by rigidity lifts to a quasi-isogeny $\rho:\BX_{\CO_K/(p)}\to X_{\CO_K/(p)}$. So $L_K$ lies in the image of the period morphism.
\end{proof}

Unfortunately an open rigid analytic subspace of $\CF^\rig$, like the image of $\breve\pi_1^\rig$, is not uniquely determined by its underlying set of points. This however is true for $K_0$-analytic spaces. Therefore it is natural to ask: What is the image of $\breve\pi_1^\an$? Examples of Genestier-Lafforgue and the author (see Example~\ref{Ex5.4} below) show that this image is in general smaller than $\CF^\an_{wa}$. As \cite[1.37 and 5.53]{RZ} indicate, Rapoport and Zink were aware of this phenomenon. We determine the image of $\breve\pi_1^\an$ in Section~\ref{Sect3}. But before we need to recall some of the rings of Fontaine Theory.

\section{Some of Fontaine's rings}

Let $K/K_0$ be a complete, rank one valued field extension, let $C$ be the completion of an algebraic closure of $K$, and let $\CO_C$ be its valuation ring. Starting out from $C$ various rings are defined in Fontaine Theory. For details on their construction see Colmez~\cite{Colmez}. We follow his notation.

We let $\wt\bE^+:=\bigl\{\,u=(u^{(n)})_{n\in\BN_0}:\es u^{(n)}\in\CO_C, (u^{(n+1)})^p=u^{(n)}\,\bigr\}$. 
With the 
multiplication $uv:=\bigl(u^{(n)}v^{(n)}\bigr)_{n\in\BN_0}$, the 
addition $u+v:=\bigl(\lim_{m\to\infty}(u^{(m+n)}+v^{(m+n)})^{p^m}\bigr)_{n\in\BN_0}$, and 
the valuation $v_\bE(u):=v_C(u^{(0)})$, $\wt \bE^+$ becomes a complete valuation ring of rank one with algebraically closed fraction field, called $\wt\bE$, of characteristic $p$.
Fix primitive $p^n$-th roots of unity $\epsilon^{(n)}$ such that $\epsilon:=(1,\epsilon^{(1)},\epsilon^{(2)},\ldots)$ is an element of $\wt\bE^+$. 

Let $\wt\bA:= W(\wt\bE)$
~and consider the automorphism $\phi=W(\Frob_p)$ of
$\wt\bA$. For an element $u\in\wt\bE$ let $[u]\in\wt\bA$ be its Teichm\"uller representative.

If $x=\sum_{i=0}^\infty p^i[x_i]\in\wt\bA$ then we set $w_k(x):=\min\{\,v_\bE(x_i):i\le k\,\}$. For $r>0$ let $\wt\bA^{(0,r]}\DS:=\bigl\{\,x\in\wt\bA:\es\lim_{k\to+\infty}w_k(x)+{\TS\frac{k}{r}}=+\infty\,\bigr\}$ and let $\wt\bB^{(0,r]}:=\wt\bA^{(0,r]}[\frac{1}{p}]$.

On $\wt\bB^{(0,r]}$ there is a valuation defined for $x=\sum_{i\gg-\infty}^\infty p^i[x_i]$ as\\
 $v^{(0,r]}(x)\es:=\es\min\{\,w_k(x)+\frac{k}{r}:\,k\in\BZ\,\}\;=\;\min\{\,v_\bE(x_i)+\frac{i}{r}:\,i\in\BZ\,\}$.

Let $\wt\bB^{]0,r]}$ be the Fr\'echet completion of $\wt\bB^{(0,r]}$ 
with respect to the family of semi-valuations $v^{[s,r]}(x):=\min\{\,v^{(0,s]}(x),v^{(0,r]}(x)\,\}$ for $0<s\le r$. The logarithm $t:=\log[\epsilon]$ converges to an element in $\wt\bB^{]0,1]}$.

Let $\wt\bB^\dagger_\rig:=\es\bigcup_{r>0}\wt\bB^{]0,r]}$. The homomorphism $\phi$ gives rise to a bicontinuous isomorphism $\phi:\wt\bB^{]0,r]}\isoto\wt\bB^{]0,r/p]}$ and thus induces an automorphism of $\wt\bB^\dagger_\rig$. It satisfies $\phi(t)=p\,t$.

Finally there is a homomorphism $\theta:\wt\bB^{(0,1]}\to C$ sending $\sum_{i\gg-\infty}^\infty p^i[x_i]$ to $\sum_{i\gg-\infty}^\infty p^ix_i^{(0)}$ which extends by continuity to $\wt\bB^{]0,1]}$. The element $t$ lies in the kernel of $\theta$.

\bigskip

\begin{definition}
A \emph{$\phi$-module over $\wt\bB^\dagger_\rig$} is a finite free $\wt\bB^\dagger_\rig$-module $\CM$ together with a $\phi$-linear automorphism $\phi_\CM:\CM\to\CM$. 
\end{definition}

The following structure theorem was proved by Kedlaya~\cite[Theorem 4.5.7]{Kedlaya}.

\begin{theorem} \label{ThmKedlaya}
Every $\phi$-module $\CM$ over $\wt\bB^\dagger_\rig$ is isomorphic to $\bigoplus_i\CM_{c_i,d_i}$ where $\CM_{c,d}=(\wt\bB^\dagger_\rig)^{\oplus d}$, $\phi_{\CM_{c,d}}=
\left(\raisebox{0.55cm}{\xymatrix @R=0pc @C=0.4pc{0 \ar@{.}[ddrr]& & p^c \\ 1 \ar@{.}[dr] \\ & 1 & 0 }}\right)\cdot\phi$ \mbox{for $c,d\in\BZ$ with $(c,d)=1$ and $d>0$.}
\end{theorem}

One defines the \emph{degree of $\CM$} as $\deg\CM:=\sum_i c_i$.

\section{The construction of $\CF^\an_a$} \label{Sect3}

As above let $D=\BD(\BX)_{K_0}$ and $\phi_D=\BD(\Frob_\BX)_{K_0}$.
Let $(\calD^{]0,1]},\phi_\calD):= (D,\phi_D)\otimes_{K_0}\wt\bB^{]0,1]}$ and consider the morphism $1\otimes\theta:\calD^{]0,1]}\to D\otimes_{K_0} C$. By a variant of a construction of Berger~\cite[\S II]{Berger04} every point $L=L_K\in\CF^\an$, with values in a field $K$ as in Section 2, defines a $\phi$-module over $\wt\bB^\dagger_\rig$ as follows.

\begin{proposition} \label{Prop3.1} (\cite[Proposition 4.1]{HartlRZ})\\
There exists a uniquely determined $\wt\bB^{]0,1]}$-submodule $t\,\calD^{]0,1]}\subset\CM_L^{]0,1]}\subset\calD^{]0,1]}$ such that $(1\otimes\theta)\bigl(\CM_L^{]0,1]}\bigr)=L_K\otimes_K C$ and $\phi_\calD:\CM_L^{]0,1]}\isoto\CM_L^{]0,1]}\otimes_{\wt\bB^{]0,1]}}\wt\bB^{]0,1/p]}$ is an isomorphism. In particular $\CM_L^{]0,1]}$ defines a $\phi$-module $\CM_L:=\CM_L^{]0,1]}\otimes_{\wt\bB^{]0,1]}}\wt\bB^\dagger_\rig$ over $\wt\bB^\dagger_\rig$.
\end{proposition}

The following results are proved in \cite{HartlRZ}.

\begin{theorem} (\cite[Theorem 4.4]{HartlRZ})\\
$\deg\CM_L=t_N(D,\phi_D,L_K)-t_H(D,\phi_D,L_K)$.
\end{theorem}

Consider the subset $\CF^\an_a:=\{\,L\in\CF^\an:\CM_L\cong\CM_{0,1}^{\oplus h}\,\}$ of the $K_0$-analytic space $\CF^\an$.

\begin{theorem} \label{Thm3.3} (\cite[Theorem 5.2]{HartlRZ})\\
The set $\CF^\an_a$ is an open $K_0$-analytic subspace of $\CF^\an_{wa}$.
\end{theorem}

\begin{remark} (on the proof of Theorem~\ref{Thm3.3}. See also Remark~\ref{Rem3.8} below.)\\
The inclusion $\CF^\an_a\subset\CF^\an_{wa}$ is seen as follows. If $D'\subset D$ is a $\phi_D$-stable $K_0$-subspace then $(D',\phi_D|_{D'},L_K\cap D_K)$ defines by Proposition~\ref{Prop3.1} a $\phi$-submodule $\CM_L'\subset\CM_L$. Since $\CM_{0,1}^{\oplus h}$ is ``semistable'' of slope zero we conclude by \cite[Lemma 3.4.8]{Kedlaya} that
\[
t_N(D',\phi_D|_{D'},L_K\cap D'_K)\;-\; t_H(D',\phi_D|_{D'},L_K\cap D'_K)\es=\es\deg\CM_L'\es\ge\es\deg\CM_L\es=\es0
\]
with equality if $D'=D$. Hence $\CF^\an_a\subset\CF^\an_{wa}$. On the other hand Berger's proof~\cite{Berger04} of the Colmez-Fontaine Theorem shows that this inclusion induces a bijection on rigid analytic points (with $K/K_0$ finite), or more generally points of $\CF^\an_{wa}$ with with values in a discretely valued field $K$.
\end{remark}

\begin{theorem} 
$\CF^\an_a$ is the image of the period morphism $\breve\pi_1^\an:\CG^\an\to\CF^\an_{wa}$.
\end{theorem}

\begin{proof} Let $L_K\in\CF^\an$ be a $K$-valued point in the image of $\breve\pi_1^\an$ with $K/K_0$ not necessarily finite. So $L_K=\BD(\rho)_K(\Lie X\dual)_K\dual$ for a Barsotti-Tate group $X$ over $\CO_K$ and a quasi-isogeny $\rho:\BX_{\CO_K/(p)}\to X_{\CO_K/(p)}$. Then the Tate module $T_pX_K$ of $X$ induces an injection $\CM_{0,1}^{\oplus h}\cong T_pX_K\otimes_{\BZ_p}\wt\bB^\dagger_\rig\hookrightarrow\CM_L$ which must be an isomorphism by reasons of degree. This was proved in \cite[Proposition 6.1]{HartlRZ}. Thus $L_K\in\CF^\an_a$. 

Conversely let $L_K\in\CF^\an_a$. Then the morphism $(\wt\bB^\dagger_\rig)^{\oplus h}\isoto\CM_{L_K}\hookrightarrow D\otimes_{K_0}\wt\bB^\dagger_\rig$ is represented, with respect to a $K_0$-basis of $D$, by a matrix $M\in M_h(\wt\bB^\dagger_\rig)$ with $tM^{-1}\in M_h(\wt\bB^\dagger_\rig)$. Then in fact $M,tM^{-1}\in M_h(\wt\bB^+_\rig)\subset M_h(\bB^+_\cris)$ by \cite[Proposition I.4.1]{Berger04b}. For notation see \cite[\S I.1]{Berger04b} or \cite[\S2]{HartlRZ}. So $M$ defines an isomorphism $\bB_\cris^{\oplus h}\isoto D\otimes_{K_0}\bB_\cris$ compatible with Frobenius, which maps $(\bB_\cris^+)^{\oplus h}$ onto the preimage of $L_K\otimes_K C$ under the map $\id\otimes\theta:D\otimes_{K_0}\bB_\cris^+\to D\otimes_{K_0}C$. This means that $(D,\phi_D,L_K)\dual$ is admissible in the sense of Faltings' \cite[Definition 1]{Faltings07}. Note that Faltings uses
 contravariant Dieudonn\'e modules. By \cite[Theorems 9 and 14]{Faltings07}, $L_K=\BD(\rho)_K(\Lie X\dual)_K\dual$ for a Barsotti-Tate group $X$ over $\CO_K$ and a quasi-isogeny $\rho:\BX_{\CO_K/(p)}\to X_{\CO_K/(p)}$, hence $L_K$ lies in the image of $\breve\pi_1^\an$.
\end{proof}

There are many examples showing that only in rare cases $\CF^\an_a=\CF^\an_{wa}$. We mention one here. Similar examples are due to A.\ Genestier and V.\ Lafforgue.

\begin{bigexample} \label{Ex5.4}
Let $D=K_0^{\oplus 5}$ and $\phi_D=\left(\raisebox{0.55cm}{\xymatrix @R=0pc @C=0.4pc{0 \ar@{.}[ddrr]& & p^{-3} \\ 1 \ar@{.}[dr] \\ & 1 & 0 }}\right)\!\cdot\phi$\,. Then $h=5\,,\, d=3$ and $\CF=\Grass(2,5)$. Since the isocrystal $(D,\phi_D)$ is simple $\CF^\an_{wa}=\CF^\an$. Let $L=L_K\in\CF^\an$. Then
\[
\CM_{-3,5}\es=\es \calD\es\supset\es\CM_L\es\supset\es t\,\calD\es\cong\es \CM_{2,5}\,.
\]
By Theorem~\ref{ThmKedlaya}, $\CM_L\cong\bigoplus_i \CM_{c_i,d_i}$ with $\sum_i d_i=\rk\CM_L=5$ and $\sum_i c_i=\deg\CM_L=0$. Moreover by \cite[Lemma 3.4.8]{Kedlaya} all the weights ${c_i}/{d_i}$ must lie between $-3/5$ and $2/5$. So either $\CM_L\cong\CM_{0,1}^{\oplus 5}$ or $\CM_L\cong\CM_{-1,2}\oplus\CM_{1,3}$. 

Now one easily checks that $\Hom_\phi\bigl(\CM_{-1,2},\calD\bigr)\;=\;\Hom_\phi\bigl(\CM_{-1,2},\CM_{-3,5}\bigr)\;=$
\[
\Biggl\{\,A\;=\;\left(  \begin{array}{cc}
\phi^5(x) & x \\ \phi^{11}(x) & \phi^6(x) \\ \phi^{17}(x) & \phi^{12}(x) \\ \phi^{23}(x) & \phi^{18}(x) \\ \phi^{29}(x) & \phi^{24}(x)   \end{array}\right)
\DS:\es x\;=\;\sum_{\nu\in\BZ}p^\nu\phi^{-10\nu}([u]),\;
u\in\bE,\;0< v_\bE(u)<\infty
\Biggr\}\,.
\]
The bad situation $\CM_L\cong\CM_{-1,2}\oplus\CM_{1,3}$ occurs if and only if $L_K$ is generated by the columns of such a matrix $\theta(A)\in C^{5\times 2}$, since then the homomorphism $A$ factors through $\CM_L$ and this forbids $\CM_L\cong\CM_{0,1}^{\oplus 5}$ by \cite[Lemma 3.4.8]{Kedlaya}.
Since obviously such $L_K$ exist, this proves that the inclusion $\CF^\an_a\subset\CF^\an_{wa}$ is strict.
\end{bigexample}

\forget{
\begin{remark} 
The example can easily be adapted to the general situation showing that only in rare cases $\CF^\an_a$ will equal $\CF^\an_{wa}$. Moreover, the question whether the two sets are equal depends largely on the combinatorics of the weights.
\end{remark}
}

\begin{remark} \label{Rem3.8} (on the proof of Theorem~\ref{Thm3.3}.)
One can explain the idea for the proof of Theorem~\ref{Thm3.3} by means of this example. Namely the complement $\CF^\an\setminus\CF^\an_a$ is the image of the continuous map from the compact set $\{u\in\wt\bE^+:\;1\le v_\bE(u)\le p^{10}\,\}$ given by $u\mapsto\theta(A)$.
\end{remark}


\textit{Acknowledgements.} The author is grateful to 
L.\ Berger, 
J.-M.\ Fontaine, V.\ Berkovich, A.\ Genestier, M.\ Kisin, V.\ Lafforgue, and M.\ Rapoport for various helpful discussions and their interest in this work.

\vfill

\noindent
Urs Hartl\\
University of M\"unster\\
Institute of Mathematics\\
Einsteinstr.~62\\
D -- 48149 M\"unster\\
Germany\\[1mm]
E-mail: 
uhartl@uni-muenster.de

\end{document}